\documentclass[11pt,reqno]{amsart}

\usepackage{filecontents}
\usepackage{microtype}
\usepackage[english]{babel}
\usepackage[T1]{fontenc}
\usepackage[utf8]{inputenc}
\usepackage[T1]{fontenc}
\usepackage{lmodern}
\usepackage{latexsym}
\usepackage{verbatim}
\usepackage{hyperref}
\usepackage{amsthm}
\usepackage{amsmath}
\usepackage{amsfonts}
\usepackage[italian]{varioref}
\usepackage{imakeidx}
\usepackage{amssymb}
\usepackage{xcolor}
\usepackage{esint}
\usepackage{accents}
\usepackage{enumitem}
\usepackage{caption} 

\usepackage{soul}

\usepackage{tikz}
\usepackage{tikz,tikz-3dplot}
\usepackage{pgfplots}
\pgfplotsset{compat=1.16}
\tikzset{
    cross/.pic = {
    \draw[rotate = 45] (-#1,0) -- (#1,0);
    \draw[rotate = 45] (0,-#1) -- (0, #1);
    }
}
\usetikzlibrary{calc, angles}
\usetikzlibrary{arrows}
\usetikzlibrary{positioning}
\usetikzlibrary{intersections}
\usepackage{pgfplots}
\pgfplotsset{compat=newest}
\usepgfplotslibrary{fillbetween}
\usepackage{pgfplots}
\pgfplotsset{compat=newest}
\usepgfplotslibrary{fillbetween}
\usetikzlibrary{patterns}
\usepackage{mathtools}
\mathtoolsset{showonlyrefs}
\usepackage{amsmath}
\allowdisplaybreaks
 
	\definecolor{ao(english)}{rgb}{0.0, 0.5, 0.0}

\usepackage{lipsum}
\usepackage{curve2e}
\definecolor{gray}{gray}{0.4}

\usepackage{amsmath}
\usepackage{amsfonts}
\usepackage{amssymb}
\usepackage{graphicx}
\usepackage{hyperref} 
\usepackage{mathrsfs} 
\usepackage{relsize} 

\theoremstyle{plain}
\newtheorem{theorem}{Theorem}[section]

\newtheorem{corollary}[theorem]{Corollary}
\newtheorem{definition}[theorem]{Definition}
\newtheorem{lemma}[theorem]{Lemma}
\newtheorem*{problem*}{Problem}

\newtheorem{proposition}[theorem]{Proposition}

\newtheorem*{theorem*}{Theorem}

\theoremstyle{definition}

\newtheorem{remark}[theorem]{Remark}

\def\Item$#1${\item $\displaystyle#1$
   \hfill\refstepcounter{equation}(\theequation)}
   
\theoremstyle{plain} 

\theoremstyle{definition}

\theoremstyle{remark} 
\numberwithin{equation}{section}
\usepackage[makeroom]{cancel}

\usepackage{bm} 
\usepackage{graphicx}
\makeindex
\usepackage{mathtools}
\usepackage{slashed}
\date{}

\usepackage{tikz,tikz-3dplot}
\tdplotsetmaincoords{80}{45}
\tdplotsetrotatedcoords{-90}{180}{-90}
\usetikzlibrary{arrows, automata, intersections}

\newcommand{\rr}{\mathbb{R}}

\newcommand{\sss}{\mathbb{S}}
\newcommand{\e}{\epsilon}

\newcommand{\dvol}{\ \textnormal{dv}}

\newcommand{\inn}{\textnormal{in}\ }
\newcommand{\onn}{\textnormal{on}\ }
\newcommand{\andd}{\textnormal{and}}

\newcommand{\dint}[1]{\textnormal{d}#1}

\newcommand{\tr}[1]{\textnormal{tr}(#1)}

\counterwithout{equation}{subsubsection}
\numberwithin{equation}{section}

\tikzstyle{mybox} = [draw=black, very thick, rectangle, rounded corners, inner ysep=5pt, inner xsep=5pt]

\title[Rigidity for two-phase overdetermined problems in $\sss^n$]{Rigidity for two-phase\\ overdetermined problems in $\sss^n$}
\author{Andrea Bisterzo and Shigeru Sakaguchi}

\address{Andrea Bisterzo
  \newline \indent Centro di Ricerca Matematica Ennio De Giorgi
  \newline \indent
Scuola Normale Superiore
\newline\indent Piazza dei Cavalieri 3, 56126, Pisa, Italy.}
\email{andrea.bisterzo@sns.it}

\address{Shigeru Sakaguchi
  \newline \indent Admissions Center
  \newline \indent
Tohoku University
\newline\indent Sendai, 980-8576, Japan.}
\email{sigersak@tohoku.ac.jp}

\begin{document}


\begin{abstract}
In this paper, we study rigidity phenomena associated with a class of overdetermined two-phase problems on the $n$-dimensional sphere $\mathbb{S}^n$. Specifically, we analyze the solutions of coupled elliptic equations defined on a domain $\Omega \subset \mathbb{S}^n$ and its complement, subject to Dirichlet boundary conditions on the interface $\partial\Omega$ and a compatibility condition on the gradient. 

We utilize two different methods. The method of moving planes works when $\partial\Omega$ is contained inside an open hemisphere. The other method works when $\Omega$ is simply connected in $\mathbb S^2$. Both lead to the same conclusion. Namely,  the existence of a solution to such overdetermined problems necessarily implies that $\Omega$ is a geodesic ball. 
In the latter, we prove this rigidity property for three distinct scenarios: a baseline problem with piecewise constant source terms, which is the two-phase version of the torsion problem; a generalization including a class of positive and regular nonlinearities; and finally an eigenvalue problem involving the first Dirichlet eigenvalue of the respective phases.
\end{abstract}

\maketitle
\tableofcontents

\section{Introduction}\label{introduction}
Let $\sss^n$ be the $n$-dimensional sphere with $n \ge 2$ and let $\Omega\subset \sss^n$  be a domain of class $C^2$ where  $\sss^n\setminus \overline{\Omega} \not=\emptyset$. Set $\Omega^+:=\Omega$ and $\Omega^-:=\sss^n\setminus \overline{\Omega}$. Denote by $\mathcal{X}_{\Omega^+}$ and $\mathcal{X}_{\Omega^-}$ the characteristic functions of the sets $\Omega^+$ and $\Omega^-$, respectively.
Set  $\sigma:= \sigma^+\mathcal{X}_{\Omega^+} + \sigma^-\mathcal{X}_{\Omega^-}$, for two  distinct positive real numbers $\sigma^+, \sigma^-$. We consider the unique solution $v=v(x,t)$ of the Cauchy problem for the  diffusion equation in the two-phase heat conductor $\sss^n$:
\begin{equation}\label{Eq_heat-Cauchy}
v_t=\mbox{ div}(\sigma\nabla v)\ \mbox{ in } \sss^n \times (0, +\infty)\ \mbox{ and } v =\mathcal{X}_{\Omega^+} \mbox{ on } \sss^n \times\{0\}.
\end{equation}
The maximum principle gives
\begin{align}\label{between 0 and 1}
0 < v < 1\ \mbox{ in }\ \sss^n\times (0,+\infty).
\end{align}
Denote by $v^+$ and $v^-$ the restrictions of $v$ to $\overline{\Omega^+}$ and $\overline{\Omega^-}$  respectively.
Then we have
\begin{align}\label{transmission condition for heat}
v^+=v^-\mbox{ and } \sigma^+\partial_\nu v^+=\sigma^-\partial_\nu v^-\ \mbox{ on }\partial\Omega\times(0,+\infty),
\end{align}
where $\nu$ denotes the outward unit normal vector to $\partial\Omega$ and \eqref{transmission condition for heat} expresses the transmission conditions on  $\partial\Omega$.
The interface $\partial\Omega$ is said to be {\it stationary isothermic} if there exists a function $a=a(t)$ for $t > 0$ satisfying
\begin{align}\label{stationary isothermic}
v(x,t) = a(t) \mbox{ for every } (x,t) \in \partial\Omega \times (0, +\infty).
\end{align}
Let $v$ be the solution of \eqref{Eq_heat-Cauchy} satisfying \eqref{stationary isothermic}. Set 
\begin{align}\label{definition of gamma}
\gamma=\frac {|\Omega^+|}{|\Omega^-|},\  f=\mathcal{X}_{\Omega^+}-\gamma\mathcal{X}_{\Omega^-} \mbox{ and }  w=(1+\gamma)\left(v-\frac {\gamma}{1+\gamma}\right).
\end{align}
Note that
\begin{align}\label{zero mean}
\int_{\sss^n} f\dvol^{\sss^n}=0
\end{align}
and $w$ is the unique solution of \eqref{Eq_heat-Cauchy} where the initial data is replaced with $f$. Then $w$ decays to $0$ as $t \to +\infty$ exponentially in $t$ and we may define the function $W=W(x)$ and the number $b$ by
\begin{align}\label{time average}
 W(x)= \int_0^\infty w(x,t)\ \dint{t}\ \mbox{  for }x \in \sss^n
\end{align}
and
\begin{align} 
\ b=\int_0^\infty\left\{(1+\gamma)\left(a(t)-\frac {\gamma}{1+\gamma}\right)\right\}\ \dint{t}.
 \end{align}
 Denote by $W^+$ and $W^-$ the restrictions of $W$ to $\overline{\Omega^+}$ and $\overline{\Omega^-}$  respectively.
Then we observe that 
\begin{align}\label{elliptic transmission problem}
-\mbox{div}(\sigma\nabla W) = f\ \mbox{ in } \sss^n,
\end{align}
and by \eqref{transmission condition for heat} 
\begin{align}\label{transmission condition for elliptic}
W^+=W^-=b\mbox{ and } \sigma^+\partial_\nu W^+=\sigma^-\partial_\nu W^-\ \mbox{ on }\partial\Omega.
\end{align}
Eventually, the function $u=u(x)$ given by  $u=\sigma(W-b)$  satisfies the following overdetermined problem
\begin{align}\label{Eq-main}
-\Delta u = \mathcal{X}_{\Omega^+}-\gamma\mathcal{X}_{\Omega^-}\ \mbox{ in }\sss^n\ \mbox{ and }\ u=0\mbox{ on } \partial\Omega.
\end{align}
Namely, the existence of $u$ solving \eqref{Eq-main} is a necessary condition for \eqref{stationary isothermic}.
Moreover, denote by $u^+$ and $u^-$ the restrictions of $u$ to $\overline{\Omega^+}$ and $\overline{\Omega^-}$  respectively. Then problem \eqref{Eq-main} is equivalent to the decoupled problems 
\begin{align}\label{Eq-decoupled problems for W}
\begin{cases}
-\Delta u^+=1 & \inn \Omega^+\\
u^+=0 & \onn \partial \Omega^+
\end{cases} \quad \andd \quad
\begin{cases}
-\Delta u^- = -\gamma & \inn \Omega^-\\
u^- =0 & \onn \partial \Omega^-
\end{cases}
\end{align}
together with the compatibility condition
\begin{align}\label{Eq-compatibility}
\nabla u^+ = \nabla u^- \quad \onn \partial \Omega\ (=\partial\Omega^+=\partial\Omega^-).
\end{align}
Note in particular that $u\in C^1(\sss^n)$ and by the maximum principle
\begin{equation}\label{positivity of upm}
  u^+ > 0 \mbox{  in } \Omega^+\ \mbox{ and  }\ u^- <  0 \mbox{ in }\Omega^-.
  \end{equation}
Thus the problem \eqref{Eq-main} resembles the classical overdetermined torsion problem
\begin{align}\label{Eq_serrin}
-\Delta u=1 \mbox{ in } \Omega,\ u=0 \mbox{ on } \partial \Omega \mbox{ and } \partial_\nu u =  \mbox{constant  on }\ \partial\Omega.
\end{align}



The main reason behind the present work is exactly the problem \eqref{Eq-main}. Rigidity phenomena for overdetermined problems associated with the
classical torsion equation have been extensively studied in geometric
analysis; see, e.g., \cite{Se71,We71,NT18} and the references therein. One of the most celebrated results, proved by J. Serrin in \cite{Se71}, states that balls are the only bounded and smooth Euclidean domains supporting a solution to \eqref{Eq_serrin}. 
Serrin's proof makes use of the method of moving planes, a very delicate technique introduced by A. D. Alexandrov in \cite{Al56} to prove that the only embedded hypersurfaces of $\rr^n$ with constant mean curvature are spheres.

Two are the key ingredients for letting the method of moving planes work. The first is to have an abundance of isometries acting on the underlying manifold, making it possible to:
	\begin{enumerate}
		\item reflect with respect to totally geodesic hypersurfaces normal to any fixed direction for which the domain is supposed to be symmetric;
		\item move such hypersurfaces toward the domain.
	\end{enumerate}
The second, the \textit{Step zero}, is to have enough space to be able to place totally geodesic hypersurfaces (again, oriented in any direction with respect to which we want to prove the symmetry of the domain) in the complement of the domain. 

Clearly, such requirements do not allow the method of moving planes to be used in arbitrary Riemannian manifolds. Some of the settings where the method, under suitable assumptions, works well are the space forms: $\rr^n$, the hyperbolic space $\mathbb{H}^n$, and $\sss^n$. Indeed, in the very interesting works \cite{M91}, by R. Molzon, and \cite{KP98}, by S. Kumaresan and J. Prajapat, the authors managed to extend both the method and Serrin's result to the case of domains contained in $\mathbb{H}^n$ or in an open hemisphere. While for $\mathbb H^n$  the rigidity theorem of Serrin can be translated \textit{ad litteram}, in $\sss^n$  things are more complicated due to the compactness of the ambient manifold, and hence the impossibility for the complement of large domains to contain great circles (i.e., totally geodesic hypersurfaces). For the same reason, the method of moving planes is not applicable to the study of the two-phase problem \eqref{Eq-main} for a general domain $\Omega\subset \sss^n$.

The problem of finding a \textit{genuine} spherical counterpart of Serrin's result has attracted considerable attention; see, among others, \cite{FM15, CV17,QX17,CV19} and the references therein. While on the one hand the obstacle in proving the rigidity theorem in $\sss^n$ seemed to be purely technical, i.e., a defect of the method of moving planes, on the other hand domains supporting a solution to \eqref{Eq_serrin} (\textit{Serrin domains}) that are different from geodesic balls were known (for instance, every \textit{isoparametric domain}), making it necessary to reformulate the problem on the sphere. In this direction, we must mention some very interesting results in \cite{CV17,CV19}: here the authors investigated some parallel problems, by adding extra assumptions on the domain or considering different equations, possibly taking into account the curvature of the manifold, in order to find one that imposes enough rigidity on the domain, hence allowing them to perfectly replicate the Euclidean result.

Regarding Serrin's problem on the sphere, in the celebrated work \cite{FMW} by M. M. Fall, I. A. Minlend, and T. Weth, the authors managed to implicitly construct, using the Crandall-Rabinowitz bifurcation theorem, a Serrin's domain obtained as a perturbation of a tubular neighbourhood of the equator. This perturbation is nontrivial, hence providing an annulus whose boundary components are diffeomorphic but not isometric to geodesic spheres. Using a similar approach, in \cite{BS25} the authors of the present paper constructed another Serrin domain in $\sss^3$, obtained as a non-trivial perturbation of an isoparametric domain with a connected toric boundary, hence showing that even under the assumption of a connected boundary, neither the class of isoparametric domains is large enough to contain all the Serrin domains of $\sss^3$.

However, in the special case of $\sss^2$ things appear more rigid and there is no room for wild phenomena among simply connected domains. Indeed, as proved by J. M. Espinar and L. Mazet in \cite{EM19}, a Serrin-like theorem on the whole sphere can be recovered in this case, showing that the dimensional threshold of two displays a deep discrepancy in the rigidity of Serrin's overdetermined problem.

The approach used by Espinar and Mazet is completely different from that adopted by Serrin. In \cite{EM19} the authors, inspired by a work of J. A. Galvez and P. Mira \cite{GM20}, first define a traceless symmetric bilinear form which measures how far the solution to \eqref{Eq_serrin} is from being radial (and hence how far the domain is from being a ball), and then, using a topological argument relying on the Poincar\'e - Hopf index theorem, conclude by proving that the only simply connected Serrin's domains in $\sss^2$ are geodesic balls. We stress that Espinar and Mazet managed to consider a wide class of semilinear PDEs, not just the torsion equation.

Some two-phase overdetermined problems in $\mathbb R^n$, which are similar to \eqref{Eq-main},  have been dealt with in \cite{KS23, DNS25}. The authors applied directly the method of moving planes to their problems to prove the rigidity theorems, where the point is to apply the method to both the interior $\Omega$ and the exterior $\mathbb R^n\setminus \overline{\Omega}$ at the same time. In their procedure, the supposition that $\Omega$ is not symmetric will lead to a contradiction to the transmission conditions on the interface $\partial\Omega$ with the aid of Hopf's boundary point lemma and Serrin's corner point lemma  (see \cite[Lemma S, p. 214]{GNN79} or \cite[Serrin's Corner Lemma, p. 393]{R97}). 

\medskip

The first result of the present paper is to show that the method of moving planes still works  to show rigidity for problem \eqref{Eq-main} in $\sss^n$ provided that $\partial\Omega$ is contained in a hemisphere.
\begin{theorem}\label{t_main_1}
Let $\Omega\subset\sss^n$ be a domain whose boundary $\partial\Omega$ is of class $C^{2,\alpha}$ for some $0 < \alpha < 1$. Assume that $\partial\Omega$ is contained in an open hemisphere and there exists a solution $u\in C^1(\sss^n)$  to  problem \eqref{Eq-main}.
Then  $\Omega$ is a geodesic ball.
\end{theorem}

\noindent
The method of moving planes also works to show rigidity for an overdetermined problem  in the hyperbolic space $\mathbb H^n$ with $n \ge 2$.
\begin{theorem}\label{t_main_2}
Let $\Omega\subset\mathbb H^n$ be a bounded domain whose boundary $\partial\Omega$ is of class $C^{2,\alpha}$ for some $0 < \alpha < 1$. Assume that there exists a solution $u\in C^1(\mathbb H^n)$  to the overdetermined problem 
\begin{align}\label{overdetermined for hyperbolic space}
\begin{cases}
-\Delta u=\mathcal X_{\Omega} & in\ \mathbb H^n,\\
 u\equiv 0 & on\ \partial \Omega,\\
 u(x) \to -c &\mbox{as } d^{\mathbb H^n}\!(x, o) \to \infty,
 \end{cases}
\end{align}
where $c$ is a  positive constant  and $d^{\mathbb H^n}\!(x, o)$ denotes the distance of $x$ from a fixed point $o$ in $\mathbb H^n$. Then $\Omega$ is a geodesic ball.
\end{theorem}
This overdetermined problem \eqref{overdetermined for hyperbolic space} is derived as follows. Let $v=v(x,t)$ be the unique bounded solution of the Cauchy problem \eqref{Eq_heat-Cauchy} where $\sss^n$ is replaced with $\mathbb H^n$ (see Corollary \ref{Cor_uniqueness_H^n}). Then by the maximum principle we have \eqref{between 0 and 1} where $\sss^n$ is replaced with $\mathbb H^n$.  Denote by $v^+$ and $v^-$ the restrictions of $v$ to $\overline{\Omega^+}$ and $\overline{\Omega^-}$  respectively. Then we also have \eqref{transmission condition for heat}. Assume that $\partial\Omega$ is stationary isothermic, that is, there exists a function $a=a(t)$ for $t > 0$ satisfying
\eqref{stationary isothermic}.  Since $\Omega$ is bounded, $v$ decays to $0$ as $t \to +\infty$ exponentially in $t$. Then we may define the function $V=V(x)$ and the number $b$ by
\begin{align}\label{time average hyperbolic space}
 V(x)= \int_0^\infty v(x,t) \ \dint{t}\ \mbox{  for }x \in \mathbb H^n\ \mbox{ and }\ b=\int_0^\infty a(t)\ \dint{t}.
 \end{align}
 Here $V(x) \to 0$ as $d^{\mathbb H^n}\!(x, o) \to \infty$, since $\Omega$ is bounded.
 Denote by $V^+$ and $V^-$ the restrictions of $V$ to $\overline{\Omega^+}$ and $\overline{\Omega^-}$  respectively.
Then we observe that 
\begin{align}\label{elliptic transmission problem hyperbolic space}
-\mbox{div}(\sigma\nabla V) = \mathcal X_\Omega\ \mbox{ in } \mathbb H^n,
\end{align}
and by \eqref{transmission condition for heat} 
\begin{align}\label{transmission condition for elliptic hyperbolic space}
V^+=V^-=b\mbox{ and } \sigma^+\partial_\nu V^+=\sigma^-\partial_\nu V^-\ \mbox{ on }\partial\Omega.
\end{align}
Eventually, the function $u=u(x)$ given by  $u=\sigma(V-b)$ satisfies \eqref{overdetermined for hyperbolic space} with $c = \sigma^-b$.

\medskip

Another main result of the present paper, which is not based on the method of moving planes, is the following rigidity theorem for  problem  \eqref{Eq-main} in $\sss^2$. It is obtained by adapting to our case the technique used in \cite{EM19}.
\begin{theorem}\label{t_main}
Let $\Omega\subset\sss^2$ be a connected and simply connected domain with $|\Omega|<|\sss^2|$. Assume that there exists a solution $u\in C^1(\sss^2)$  to problem  \eqref{Eq-main}.
Then $\Omega$ is a geodesic ball.
\end{theorem}

As in \cite{EM19}, we get a more general version of the above theorem. To avoid any technicality that could make things unclear to the reader, we prefer to provide a detailed proof of the rigidity result only in the case of Theorem \ref{t_main}, and then to highlight the differences needed to prove the following generalization.
\begin{theorem}\label{t_general}
Let $\Omega\subset\sss^2$ be a connected and simply connected domain with $|\Omega|<|\sss^2|$ and denote $\Omega^+:=\Omega$ and $\Omega^-:=\sss^2\setminus \overline{\Omega}$. Consider a positive $C^1$ function $f$  satisfying the inequality
\begin{align}\label{Eq-condition-f}
f(x)\geq xf^\prime(x) \quad \mbox{ for every } x > 0.
\end{align}
For $R:=\cos^{-1}\left(1-\frac{|\Omega^+|}{2\pi}\right)$, i.e. $R>0$ is so that $|B_R|=|\Omega^+|$, and $\gamma>0$,
assume that there exists a solution $u\in C^1(\sss^2)$  to the overdetermined problem
\begin{align}\label{Eq-general}
\begin{cases}
-\Delta u=f(u)\chi_{\Omega^+}-\gamma f\left(-\frac{u}{\gamma}\right) \chi_{\Omega^-} & in\ \sss^2,\\
u\equiv 0 & on\ \partial \Omega,
\end{cases}
\end{align}
Then $\Omega$ is a geodesic ball.
\end{theorem}

The last theorem we present involves the eigenvalue problem, which does not fall into the class of nonlinearities presented in Theorem \ref{t_general}.
\begin{theorem}\label{t_eigenvalue}
Let $\Omega\subset\sss^2$ be a connected and simply connected domain with $|\Omega|<|\sss^2|$ and denote $\Omega^+:=\Omega$ and $\Omega^-:=\sss^2\setminus \overline{\Omega}$. For the first Dirichlet eigenvalue $\lambda(U)$ of the Laplace-Beltrami operator on the domain $U\subset \sss^2$, assume that there exists a nontrivial solution $u\in C^1(\sss^2)$ to the overdetermined problem
\begin{align}\label{Eq-eigenvalue}
\begin{cases}
-\Delta u=\Big[\lambda(\Omega^+)\chi_{\Omega^+}+\lambda(\Omega^-) \chi_{\Omega^-}\Big]u & in\ \sss^2,\\
u\equiv 0 & on\ \partial \Omega.
\end{cases}
\end{align}
Then $\Omega$ is a geodesic ball.
\end{theorem}

\noindent\textbf{Structure of the paper.} Section \ref{Sec-t-main_1_2} is devoted to the proofs of Theorems \ref{t_main_1} and \ref{t_main_2}. It is  explained how to apply the method of moving planes to our problems in detail. In sSction \ref{Sec-t-main} we prove Theorem \ref{t_main} by adapting the technique in \cite{EM19}  to our problem. However, there is a major difference between our problem and the Serrin problem (\cite[Corollary 4.3, p. 2069]{EM19}), which comes from the difference of overdetermination. Classes of candidate  solutions in subsection \ref{candidates solutions for our problem} and the continuity of the form $Q$ on the interface in Lemma \ref{l_Q_continuity}  together with Lemma \ref{l_isolated_singularities} are specific to our problem. In section \ref{Proofs of Theorems 1.4 and 1.5}, we explain how to modify the proof of Theorem \ref{t_main} to prove Theorems \ref{t_general} and \ref{t_eigenvalue}. The Appendix is devoted to proving the uniqueness of the bounded solution to the Cauchy problem for the diffusion equation in the two-phase heat conductor $\mathbb H^n$.


\section{Proofs of Theorems \ref{t_main_1} and \ref{t_main_2}}\label{Sec-t-main_1_2}

Let us first prove Theorem \ref{t_main_1}. Let $\sss^n_+$ be the open upper hemisphere containing the north pole $\mathcal N$. Without loss of generality, we assume that  $\partial\Omega \subset \sss^n_+$. Then either $\overline{\Omega} \subset \sss^n_+$ or $\sss^n \setminus \Omega \subset \sss^n_+$. The proof for the latter case is similar to that of the former case, or if $\sss^n \setminus \Omega \subset \sss^n_+$, we may consider  the function $-\frac 1\gamma u$ with  $\sss^n \setminus \Omega$ instead of $u$ with $\Omega$. Let us consider the former case and apply the method of moving planes to $\overline{\Omega}\subset  \sss^n_+$. With the arguments of  \cite{KS23, DNS25} in mind, we follow \cite[Proof of Theorem 4, pp. 150--156]{M91}. 

Let $u$  be the solution to the overdetermined problem \eqref{Eq-main}. Since $\partial\Omega$ is of class $C^{2,\alpha}$, it follows from the regularity theory of elliptic partial differential equations (see \cite{GT83}), $u^+ \in C^{2,\alpha}(\overline{\Omega^+})$ and $u^-\in C^{2,\alpha}(\overline{\Omega^-})$, where $u^+, u^-$ are given in \eqref{Eq-decoupled problems for W}. 

Let $\gamma: t \mapsto \gamma(t)$ be a geodesic in $\sss^n$ through $\mathcal N$ with arclength parameter,  and for each $t$ let $H_t$ denote the totally geodesic hypersurface through the point $\gamma(t)$ which is orthogonal to the unit tangent vector $\gamma^\prime(t)$. For each $t$, let $M^-_t$ be the open set of points in $\sss^n$ lying on the side of $H_t$ which contains $\{ \gamma(s)\, :\, t -\pi< s < t\}$ and let $M^+_t$ be that on the other side of $H_t$. Thus $\sss^n=M^-_t\sqcup H_t \sqcup M^+_t$. Denote by $R_t$ be the isometry of $\sss^n$ defined by the reflection in $H_t$. 

Since $\overline{\Omega}\subset  \sss^n_+$, there exists $t_0$ with $H_{t_0}\cap \overline{\Omega}=\emptyset$. For $t \in (t_0, t_0+\pi)$ with $H_t\cap\Omega\not=\emptyset$, set $\Omega_t = \Omega \cap M^-_t$ and let
$\Omega^t$ be the image of $\Omega_t$ under $R_t$.
We infer that there exists  $t_1 \in (t_0, t_0+\pi)$ such that
\begin{align}\label{moving planes}
 \Omega_{t_1}\not=\emptyset\ \mbox{ and } \ \Omega^{t_1} \subset \Omega.
 \end{align}
Then, as $t$ increases, $\Omega^t$ remains in $\Omega$  until one of the following two events occurs for some $t_*\in [t_1,  t_0+\pi)$:
\begin{itemize}
\item[ (i)]\ $\Omega^{t_*}$ becomes internally tangent to $\partial\Omega$ at some point $P \in \partial\Omega\setminus \overline{\Omega_{t_*}}$;
\item[ (ii)]\ $H_{t_*}$ reaches a position where it is orthogonal to $\partial\Omega$ at some point $Q \in \partial\Omega \cap H_{t_*}$, and for every $t \in (t_0, t_*)$ with $H_t\cap\Omega\not=\emptyset$, $H_t$ is not orthogonal to $\partial\Omega$ at all points in $\partial\Omega\cap H_t$. 
\end{itemize}
We claim that $\Omega$ is symmetric with respect to $H_{t_*}$. Suppose that $\Omega$ is not symmetric with respect to $H_{t_*}$. Denote by $D$ the image of $\left(\sss^n\setminus\overline{\Omega}\right)\cap M^-_{t_*}$ under $R_{t_*}$. Let $\Sigma$ be the connected component of $\left(\sss^n\setminus\overline{\Omega}\right)\cap M^+_{t_*}$ whose boundary contains the points $P$ and $Q$ in the respective cases (i) and (ii). Since $\Omega^{t_*}\subset\Omega$, we notice that
\begin{align}
\Sigma\subset\left(\sss^n\setminus\overline{\Omega}\right)\cap M^+_{t_*} \subset D.
\end{align}
Let us introduce the two functions $U^\pm=U^\pm(x)$ by
\begin{align}\label{two auxiliary functions}
U^+(x) := u^+(x)-u^+(R_{t_*}x)\ &\mbox{ for } x \in \overline{\Omega^{t_*}}, \\
U^-(x) := u^-(x)-u^-(R_{t_*}x)\ &\mbox{ for } x \in \overline{\Sigma}.
\end{align}
Then it follows from \eqref{Eq-decoupled problems for W} and \eqref{positivity of upm} that 
\begin{align}\label{harmonicity of Upm}
&\Delta U^+= 0 \mbox{ in } \Omega^{t_*}\ \mbox{ and }\ U^+ \ge 0 \mbox{ on } \partial \Omega^{t_*}, \\
&\Delta U^-= 0 \mbox{ in } \Sigma\ \quad \mbox{ and }\ U^- \ge 0 \mbox{ on } \partial \Sigma,
\end{align}
and hence by the maximum principle
\begin{equation}\label{positivity of Upm}
U^+ > 0 \mbox{ in } \Omega^{t_*}\ \mbox{ and }\ U^- > 0 \mbox{ in } \Sigma,
\end{equation}
since $U^+\not\equiv 0$ on $\partial \Omega^{t_*}$ and $U^- \not\equiv 0$ on $\partial \Sigma$ because of our assumption that $\Omega$ is not symmetric with respect to $H_{t_*}$.

\medskip
Let us first consider the case (i). Since $U^+(P)=U^-(P)=0$ from \eqref{Eq-decoupled problems for W}, it follows from \eqref{positivity of Upm} and Hopf's boundary point lemma that
\begin{equation}\label{normal derivatives at p}
\partial_\nu U^+(P) < 0 < \partial_\nu U^-(P),
\end{equation}
where we used the fact that $\nu$ is the outward unit normal vector to $\partial\Omega$ at $P$ as well as the inward unit normal vector to $\partial\Sigma$ at $P$.
Since reflection symmetry yields that
\begin{equation}\label{conclusion of reflection symmetry}
\partial_\nu\left(u^+(R_{t_*}x)\right)\big|_{x=P} = \partial_\nu u^+(R_{t_*}P)\ \mbox{ and }\ \partial_\nu\left(u^-(R_{t_*}x)\right)\big|_{x=P} = \partial_\nu u^-(R_{t_*}P),
\end{equation}
we have from \eqref{normal derivatives at p} that
\begin{equation}\label{pre conclusion}
\partial_\nu u^+(P) <  \partial_\nu u^+(R_{t_*}P)\ \mbox{ and }\ \partial_\nu u^-(P) >  \partial_\nu u^-(R_{t_*}P),
\end{equation}
which contradict the fact that $\partial_\nu u^+= \partial_\nu u^-$ on $\partial\Omega$ due to $u \in C^1(\sss^n)$.

\medskip
Let us proceed to the case (ii). We introduce the stereographic projection $\Pi: \sss^n \to \mathbb R^n = \{ (x,0) \in \mathbb R^{n+1}\, :\, x \in \mathbb R^n \}$ as in \cite[p. 1105]{BBF98}.
Then the Laplace-Beltrami operator $\Delta$ in $\sss^n$ is translated into $p^{-n}\mbox{div}(p^{n-2}\nabla \cdot)$ with $p = p(x) = \frac 2{1+|x|^2}$ for $x \in \mathbb R^n$.
With some rotation of $\sss^n$ we may assume that $\Pi(Q)=0,\ \Pi(H_{t_*})= \{ x_1=0 \},\ \Pi(M^-_{t_*})=\{ x_1>0\}$ and the boundary $\partial\Omega$ in a neighbourhood of $Q$  is represented by $x_n=\psi(\hat{x})$ where $\hat{x}=(x_1,\dots,x_{n-1}) \in \mathbb R^{n-1}$ in a neighborhood of the origin in $\mathbb R^n$ for some $C^{2,\alpha}$ function
$\psi : \mathbb R^{n-1} \to \mathbb R$ satisfying
\begin{equation}\label{facts obtained by rotation}
 \psi(0)=0, \nabla\psi(0)=0, \mbox{ and } \partial_\nu=\frac 1{p(x)\sqrt{1+|\nabla\psi|^2}}\left(-\sum_{j=1}^{n-1}\partial_{x_j}\psi\partial_{x_j} +\partial_{x_n}\right).
 \end{equation}
 Notice that, since $R_{t_*} x = (-x_1,x_2,\dots, x_n)$ for $x \in \mathbb R^n$, 
 \begin{align}
 &p(x)=p(R_{t_*} x), \label{from x1=0 1}\\
 &U^\pm(x)=u^\pm(x_1,x_2,\dots,x_{n}) - u^\pm(-x_1,x_2,\dots,x_{n}).\label{from x1=0 2}
 \end{align}
 Since event (ii) occurs, we observe that the function $\partial_{x_1}\psi(0, x_2, \dots, x_{n-1})$ achieves its local maximum $0$ at $(x_2,\dots,x_{n-1}) = 0 \in \mathbb R^{n-2}$, and hence 
 \begin{equation}\label{when n ge 3}
 \partial^2_{x_1 x_j}\psi(0) = 0\ \mbox{ for } j=2, \dots, n-1\ \mbox{ when } n \ge 3.
 \end{equation}
At every point $(\hat{x}, \psi(\hat{x}))$ in a neighborhood of the origin, we have three equalities:
\begin{equation}\label{three equalities}
u^\pm=0 \mbox{ and } -\!\sum_{k=1}^{n-1}\partial_{x_k}\psi\partial_{x_k}u^++\partial_{x_n}u^+= -\!\sum_{k=1}^{n-1}\partial_{x_k}\psi\partial_{x_k}u^-+\partial_{x_n}u^-.
\end{equation}
Differentiating the first two equalities of \eqref{three equalities} in $x_i$ for $i=1,\dots, n-1$ yields that at $(\hat{x}, \psi(\hat{x}))$
\begin{equation}\label{1st derivatives}
\partial_{x_i}u^\pm+\partial_{x_n}u^\pm\partial_{x_i}\psi =0,
\end{equation}
and differentiating this again in $x_j$ for $j=1,\dots, n-1$ yields that at $(\hat{x}, \psi(\hat{x}))$
\begin{equation}\label{2nd derivatives further}
\partial^2_{x_jx_i}u^\pm+\partial^2_{x_nx_i}u^\pm\partial_{x_j}\psi+\partial^2_{x_jx_n}u^\pm\partial_{x_i}\psi+\partial^2_{x_n}u^\pm\partial_{x_i}\psi\partial_{x_j}\psi+\partial_{x_n}u^\pm\partial^2_{x_jx_i}\psi =0.
\end{equation}
By letting $x=0$ in these equalities, we infer that
\begin{align}\label{vanishing of first and second derivatives}
&\partial_{x_i}u^\pm(0)=0 \mbox{ for } i=1,\dots,n-1,\\
&\partial^2_{x_1x_j}u^\pm(0) = 0 \mbox{ for } j=2,\dots,n-1\mbox{ when }n \ge 3,
\end{align}
where we used \eqref{when n ge 3} when $n\ge 3$. Then, differentiating the last equality of \eqref{three equalities} in $x_i$ for $i=1,\dots, n-1$ and letting $x=0$ give
\begin{equation}\label{further second derivatives}
\partial^2_{x_ix_n}u^+(0)=\partial^2_{x_ix_n} u^-(0)\ \mbox{ for } i=1,\dots, n-1.
\end{equation}
Thus we have from \eqref{from x1=0 2} and \eqref{vanishing of first and second derivatives} that
\begin{align}
U^\pm(0) = \partial_{x_i}U^\pm(0) = 0\  &\mbox{ for } i=1,\dots, n, \label{vanishing important1}\\
\partial^2_{x_n} U^\pm(0) = \partial^2_{x_1x_j}U^\pm(0) =0\  &\mbox{ for } j=1,\dots, n-1. \label{vanishing important2}
\end{align}
Since  $\mbox{ div}(p^{n-2}\nabla U^\pm) = 0$, by virtue of \eqref{positivity of Upm} and \eqref{vanishing important1}, we can apply Serrin's corner point lemma  (see \cite[Lemma S, p. 214]{GNN79} or \cite[Serrin's Corner Lemma, p. 393]{R97}) to $U^\pm$ and show that  letting $s_\pm=(-1,0,\dots,0,\mp 1) \in \mathbb R^n$ gives
\begin{equation}\label{Serrin's corner point lemma applies}
\partial^2_{s_+}U^+(0) > 0 \mbox{ and } \partial^2_{s_-}U^-(0) > 0,
\end{equation}
 where we used the fact that each of the directions $s_\pm$ enters $\Omega^{t_*}$ and $\Sigma$, respectively, transversally to both of the hypersurfaces $\partial\Omega$ and $H_{t_*}$ at the origin. Hence by \eqref{vanishing important2}
 \begin{equation}\label{second derivatives by upm}
 \partial^2_{s_\pm} U^\pm(0) = \pm 2\partial^2_{x_1x_n}U^\pm(0) = \pm4\partial^2_{x_1x_n}u^\pm(0).
 \end{equation}
 Consequently, combining \eqref{Serrin's corner point lemma applies} with \eqref{second derivatives by upm} yields that
 \begin{equation}\label{final contradiction}
 \partial^2_{x_1x_n}u^-(0) < 0 < \partial^2_{x_1x_n}u^+(0),
 \end{equation}
 which contradicts \eqref{further second derivatives} with $i=1$. 
 
 It remains to prove Theorem \ref{t_main_2}. The proof follows basically along the same lines. We may first replace the number $\pi$ with $+\infty$ in defining $M_t^-$. In the case (ii), we may start with  $\mathbb H^n$, the Poincar\'e ball model given by the unit ball $\mathbb B^n=\{ x \in \mathbb R^n\, :\, |x| < 1 \}$ endowed with the metric  $g= p^2(x)|dx|^2$ and  the Laplace-Beltrami operator $\Delta=p^{-n}\mbox{div}(p^{n-2}\nabla \cdot)$ with $p(x) = \frac 2{1-|x|^2}$ for $x \in \mathbb B^n$ (see \cite[p. 1105]{BBF98}). That is, we may replace $\frac 2{1+|x|^2}, \mathbb R^n$ with $\frac 2{1-|x|^2}, \mathbb B^n$, respectively.


\section{Proof of Theorem \ref{t_main}}\label{Sec-t-main}

Let $\Omega\subset \sss^2$ be a connected and simply connected domain of class $C^2$, and let 
$u$ be the solution to the overdetermined problem \eqref{Eq-main}.

\subsection{Classes of candidate solutions}\label{candidates solutions for our problem}
If $\Omega=B_R(o)$, where $R\in (0,\pi)$ and $o\in \sss^2$ are fixed, one has that $\gamma:=\frac{|\Omega^+|}{|\Omega^-|}=\frac{1-\cos(R)}{1+\cos(R)}$ and the solution $v$ to \eqref{Eq-main} is radial, i.e., it only depends on the distance $r(\cdot):=d^{\sss^2}(o,\cdot)$ from $o$:
\begin{align}
v^+(x)=\widehat{v}^+(r(x)) \quad \andd \quad v^-(x)=\gamma \widehat{v}^-(r(x)),
\end{align}
where
\begin{align}
\widehat{v}^+(r):=\ln\left(\frac{1+\cos(r)}{1+\cos(R)}\right) \quad \andd \quad \widehat{v}^-(r):=\ln\left(\frac{1-\cos(R)}{1-\cos(r)}\right).
\end{align}
In polar coordinates $(r,\theta)$ centred at $o$
\begin{align}\label{Eq_gradient_v}
\nabla v^+=-\frac{\sin(r)}{1+\cos(r)}\partial_r \quad \andd \quad \nabla v^-=-\gamma\frac{\sin(r)}{1-\cos(r)}\partial_r.
\end{align}

Following the notation of \cite{EM19}, we denote by $\mathcal{C}$ the following family of couples of $C^3$ functions
\begin{align}
\mathcal{C}:=\{V:=(v^+,v^-)=(v^+_{p_+,w_+,a_+},v^-_{p_-,w_-,a_-})\}_{(p_\pm,w_\pm,a_\pm)\in T\sss^2\times \rr^\pm}
\end{align}
solving respectively
\begin{align}
\begin{cases}
-\Delta v^+=1 & \inn B_{R^+}(q^+),\\
-\Delta v^-=-\gamma & \inn B_{R^-}(q^-),\\
v^\pm=0 & \onn \partial B_{R^\pm}(q^\pm),\\
p_\pm\in \overline{B_{R^\pm}(q^\pm)},\\
v^\pm(p_\pm)=a_\pm,\\
\nabla v^\pm(p_\pm)=w_\pm
\end{cases}
\end{align}
for some $R^\pm=R^\pm(p_\pm,w_\pm,a_\pm)>0$ and $q^\pm=q^\pm(p_\pm,w_\pm,a_\pm)\in \sss^2$.
Note that each $v^\pm_{p_\pm,w_\pm,a_\pm}$ is radial, in the sense that it only depends on the distance from the point $q^\pm_{p_\pm,w_\pm,a_\pm}$. To simplify notation, the subscript $(p_\pm, w_\pm, a_\pm)$ is suppressed unless necessary for clarity. The existence of such functions is obtained in \cite[Section 3]{EM19}.


For $\gamma>0$, we denote
\begin{align}
\widetilde{\mathcal{C}}:=&\Big\{V\in \mathcal{C}\ :\ |B_{R^+}(q^+)|=\gamma |B_{R^-}(q^-)| \quad \andd \quad q^\pm \ \textnormal{are antipodal}\Big\}\\
=&\Bigg\{V\in \mathcal{C}\ :\ R^+=\cos^{-1}\left(\frac{1-\gamma}{1+\gamma}\right), \ R^-=\pi-R^+\ \andd \ q^\pm \ \textnormal{are antipodal}\Bigg\}.
\end{align}
Hence, for any $V\in \widetilde{\mathcal{C}}$ the geodesic balls $B_{R^\pm}(q^\pm)$ have common boundary $\Sigma$ and
\begin{align}
B_{R^-}(q^-)\sqcup B_{R^+}(q^+) \sqcup \Sigma = \sss^2.
\end{align}

For any $x\in \sss^2$ let $Q$ be the following symmetric bilinear form
\begin{align}\label{Eq-Q-def}
Q_x:=\chi_{\overline{\Omega^+}}(x) Q^+_x + \chi_{\Omega^-}(x)Q^-_x,
\end{align}
where
\begin{align}
&Q^+_x:=\nabla^2 u^+(x) - \nabla^2 v^+_{x,\nabla u^+(x), u^+(x)}(x),\\
&Q^-_x:=\nabla^2 u^-(x) - \nabla^2 v^-_{x,\nabla u^-(x), u^-(x)}(x).
\end{align}
In what follows, when there is no risk of confusion, we simply denote 
\begin{align}
v^\pm_x:=v^\pm_{x,\nabla u^\pm (x), u^\pm (x)}.
\end{align}

\subsection{Proof of Theorem \ref{t_main}}

Before proving Theorem \ref{t_main} we need to prove some preliminary properties satisfied by the form $Q$. The first step is the next lemma.
\begin{lemma}\label{l_Q_continuity}
The form $Q$ is continuous in $\sss^2$.
\end{lemma}
\begin{proof}
If $x$ is in $\Omega^+$ or in $\Omega^-$, then the continuity follows as in \cite{EM19}. Hence, fix $x\in \partial \Omega$ and denote by $\tau$ and $\eta$ the tangent vector and the (outward) unit normal vector to $\partial \Omega$ respectively. Observing that $Q^-$ continuously extends to $\partial \Omega$, we only have to prove that $Q^+$ and $Q^-$ are equal at $x$. Indeed,
\begin{itemize}
\item \underline{$Q^+_x(\tau,\eta)=Q^-_x(\tau,\eta)$}: since $v^+$ and $v^-$ only depend on the distance $r^\pm(\cdot)=d^{\sss^2}(\cdot, q^\pm)$ from $q^\pm$ respectively, and since $\nabla r^\pm (x)\parallel \eta(x)$, then
\begin{align}\label{l_cont_app1}
\nabla^2 v^+(x)(\tau(x),\eta(x))=\nabla^2 v^-(x)(\tau(x),\eta(x))=0.
\end{align}
Moreover, since $\nabla u^+=\nabla u^-$ on $\partial \Omega$, it follows that along $\partial \Omega$
\begin{align}
\nabla u^+ \cdot \eta=\nabla u^- \cdot \eta
\end{align}
and hence
\begin{align}
\tau(\nabla u^+ \cdot \eta)=\tau(\nabla u^- \cdot \eta)
\end{align}
implying
\begin{align}
\nabla^2 u^+(\tau,\eta)+\nabla u^+ \cdot \nabla_\tau \eta=\nabla^2 u^-(\tau,\eta)+\nabla u^- \cdot \nabla_\tau \eta.
\end{align}
Since $u^+\equiv u^-\equiv 0$ at $\partial \Omega$ and $\nabla_\tau \eta \parallel \tau$, we get
\begin{align}
\nabla^2 u^+(\tau,\eta)=\nabla^2 u^-(\tau,\eta).
\end{align}
This latter equality, together with \eqref{l_cont_app1}, implies
\begin{align}
Q_x^+(\tau,\eta)=Q^-_x(\tau,\eta).
\end{align}

\item \underline{$Q^+_x(\tau,\tau)=Q^-_x(\tau,\tau)$}: since $\nabla u^\pm \bot \tau$ along $\partial \Omega$
\begin{align}
\nabla u^+ \cdot \tau=0=\nabla u^- \cdot \tau
\end{align}
and hence
\begin{align}
\tau(\nabla u^+ \cdot \tau)=\tau(\nabla u^- \cdot \tau).
\end{align}
Namely, 
\begin{align}
\nabla^2 u^+(\tau,\tau)+\nabla u^+\cdot \nabla_\tau \tau=\nabla^2 u^-(\tau,\tau)+\nabla u^-\cdot \nabla_\tau \tau.
\end{align}
Hence, using $\nabla u^+=\nabla u^-$ we have
\begin{align}
\nabla^2 u^+(\tau,\tau)=\nabla^2 u^-(\tau,\tau).
\end{align}
This, together with the fact that $v^\pm$ only depend on $r^\pm$ respectively, and hence
\begin{align}
\nabla^2 v^\pm_x(x)(\tau(x),\tau(x))=0,
\end{align}
which implies that
\begin{align}
Q^+_x(\tau,\tau)=\nabla^2 u^+(\tau,\tau)=\nabla^2 u^-(\tau,\tau)=Q^-_x(\tau,\tau).
\end{align}

\item \underline{$Q^+_x(\eta,\eta)=Q^-_x(\eta,\eta)$}: we start observing that for every $y\in \sss^2$ one has
\begin{align}
\tr{Q_y}=&\chi_{\overline{\Omega^+}}(y)\left(\Delta u^+(y)-\Delta v^+_y(y)\right)\\
&+\chi_{\Omega^-}(y)\left(\Delta u^-(y)- \Delta v^-_y(y)\right)\\
=& \chi_{\overline{\Omega^+}}(y)(-1+1)+\chi_{\Omega^-}(y)(\gamma-\gamma)\\
=&0
\end{align}
and hence $\tr{Q^\pm}=0$. This implies that
\begin{align}
Q^\pm_x(\eta,\eta)=-Q^\pm_x(\tau,\tau)
\end{align}
and hence, by previous point,
\begin{align}
Q^+_x(\eta,\eta)=-Q^+_x(\tau,\tau)=-Q^-_x(\tau,\tau)=Q^-_x(\eta,\eta).
\end{align}
\end{itemize}
It follows that $Q^+$ and $Q^-$ coincide along $\partial \Omega$ and hence $Q$ is continuous.
\end{proof}
\begin{remark}\label{Rmk-Continuity-Q}
We stress that in the previous proof we only used the compatibility condition \eqref{Eq-compatibility} and the radiality of the candidate solutions. Hence, an analogue of Lemma \ref{l_Q_continuity} can be proved for the other overdetermined problems \eqref{Eq-general} and \eqref{Eq-eigenvalue} as soon as these properties are satisfied, together with the condition that for every $x\in \sss^2$ there exists a candidate solution at $x$, that is, $v^\pm_{x,\nabla u^\pm (x), u^\pm(x)}\in \mathcal{C}$.
\end{remark}

Fix a local conformal parametrization $\psi:U\to V$ of a domain $V\Subset \sss^2$, where $U\subset \mathbb{C}$. Without loss of generality, we can assume $0\in U$. Denote $V^\pm:=V\cap \Omega^\pm$, parametrized respectively by two disjoint open sets $U^\pm\subset U$ where $U^\pm=\{ z \in U\, :\, \pm\mbox{ Im}(z) > 0 \}$ and  $\overline{U}=\overline{U^+}\cup U^-$ . Let $\mathfrak{P}$ denote the $(2,0)$ part of $Q|_U$: since $Q$ is symmetric, then
\begin{align}
Q=\mathfrak{P}+\overline{\mathfrak{P}}.
\end{align}
Moreover, if $z$ denotes the local conformal parameter,
\begin{align}
\mathfrak{P}=P(z) \dint{z}^2
\end{align}
for $P(z)=Q\left(\partial_z,\partial_z\right)$. Similarly, one can define the $(2,0)$ parts $\mathfrak{P}^+$ and $\mathfrak{P}^-$ of $Q^+|_{\overline{U^+}}$ and $Q^-|_{U^-}$ respectively, obtaining that in the local conformal coordinate $z$
\begin{align}
\mathfrak{P}^+=P^+(z) \dint{z}^2 \quad \andd \quad \mathfrak{P}^-=P^-(z) \dint{z}^2
\end{align}
and hence that
\begin{align}
\mathfrak{P}=\chi_{\overline{U^+}} \mathfrak{P}^+ +\chi_{U^-} \mathfrak{P}^- 
\end{align}
or, equivalently,
\begin{align}
P=\chi_{\overline{U^+}} P^+ +\chi_{U^-} P^-
\end{align}
which, by the continuity of $Q$, is continuous.

\begin{remark}\label{rmk_P}
By \cite[Claim C]{EM19} applied to $Q^+$ and $Q^-$ as defined on $\Omega^+$ and $\Omega^-$ respectively, there exist two complex functions $\beta^+$ and $\beta^-$ so that
\begin{align}
\partial_{\overline{z}} P^+(z)=\beta^+(z) \overline{P^+}(z) \quad \andd \quad \partial_{\overline{z}} P^-(z)=\beta^-(z) \overline{P^-}(z).
\end{align}
Set
\begin{align}
\beta:=\chi_{U^+}\beta^+ + \chi_{U^-} \beta^-.
\end{align}
Then the previous identities imply that for every $U'\Subset U$
\begin{align}
|\partial_{\overline{z}} P(z)|&= |\beta(z)| |P(z)|\\
&\leq C |P(z)|\quad \textnormal{almost everywhere in } U',
\end{align}
where $C:=\sup_{\overline{U'}} |\beta|$. Up to change the notation, we can suppose $U'=U$.
\end{remark}

Next lemma describes the structure of $P$ in a neighbourhood of any of its isolated zeros.
\begin{lemma}\label{l_isolated_singularities}
There exists a holomorphic function $h$ and a H\"older continuous function $s$ so that $P(z)=e^{is(z)}h(z)$ in $U$.
\end{lemma}
\begin{proof}
Let us define
\begin{align}
\sigma(z):=\begin{cases}
\frac{\partial_{\overline{z}}P(z)}{P(z)} & \inn U\setminus\psi^{-1}\left(V\cap \partial\Omega\right),\\ 0 & \inn U\cap \psi^{-1}\left(V\cap \partial \Omega\right).
\end{cases}
\end{align}
By Remark \ref{rmk_P}
\begin{align}
|\sigma|\leq C \quad \inn U
\end{align}
and hence, for $\zeta=\xi+i \eta$, the following integral function is well defined
\begin{align}
s(z):=-\frac{1}{\pi} \int_{B_R(0)} \sigma(\xi) \left(\frac{1}{\zeta-z}-\frac{1}{\zeta}\right)\ \dint{\xi}\, \dint{\eta},
\end{align}
where $R>0$ is small enough so that $B_R(0)\subset U$. In particular, observe that
\begin{align}
\partial_{\overline{z}}s=\sigma \quad \inn U^+\sqcup U^-.
\end{align}
Indeed, for $z\in U^+$ fixed, $\e>0$ small enough and $z_0\in U^+$ so that $z\in B_{\e/2}(z_0)$ and $B_\e(z_0)\subset U^+$, we have
\begin{align}
s(z)=&-\frac{1}{\pi} \int_{B_R(0)\setminus B_\e(z_0)}\sigma(\xi) \left(\frac{1}{\zeta-z}-\frac{1}{\zeta}\right)\ \dint{\xi}\, \dint{\eta}\\
& - \frac{1}{\pi}\int_{B_\e(z_0)} \sigma(\xi) \left(\frac{1}{\zeta-z}-\frac{1}{\zeta}\right)\ \dint{\xi}\, \dint{\eta}
\end{align}
which implies that
\begin{align}
\partial_{\overline{z}}s(z)=&-\frac{1}{\pi}\underbrace{\partial_{\overline{z}}\left(\int_{B_R(0)\setminus B_\e(z_0)}\sigma(\xi) \left(\frac{1}{\zeta-z}-\frac{1}{\zeta}\right)\ \dint{\xi}\, \dint{\eta}\right)}_{=0}\\
&-\frac{1}{\pi}\partial_{\overline{z}}\left(\int_{B_\e(z_0)} \sigma(\xi) \left(\frac{1}{\zeta-z}-\frac{1}{\zeta}\right)\ \dint{\xi}\, \dint{\eta}\right)\\
=&-\frac{1}{\pi}\partial_{\overline{z}}\left(\int_{B_\e(z_0)} \sigma(\xi) \left(\frac{1}{\zeta-z}-\frac{1}{\zeta}\right)\ \dint{\xi}\, \dint{\eta}\right)\\
=&\sigma(z).
\end{align}
Similarly,  $\partial_{\overline z}s=\sigma$ for $z\in U^-$. Arguing as in \cite[Lemma 2.7.1]{Jo91}, we obtain that the function
\begin{align}
h:=Pe^{-s}
\end{align}
is holomorphic in $U^+\sqcup U^-$. Moreover, by the definition of $\sigma$ and by Remark \ref{rmk_P}, $h$ is continuous in $U$. It remains to prove that $h$ is holomorphic in the whole $U$. To this aim, thanks to Morera's theorem, we only have to prove that the line integral of $h$ on any rectangle $R\subset \Omega$ is zero. 
Hence, let $R$ be a rectangle intersecting $\partial \Omega \cap U$: $R=\partial\left([a,b]\times [-c,d]\right)$ for $-\infty<a<b<+\infty$ and $c,d>0$. For $\e>0$ small enough, denote the rectangles
\begin{align}
& R^+_\e:= \partial\left([a, b]\times [\e,d] \right)\subset U^+,\\
&R^-_\e:=\partial \left([a,b] \times [-c,-\e]\right) \subset U^-
\end{align}
so that
\begin{itemize}
\item $R^\pm_\e$ converges uniformly to $R^\pm_0$;
\item $R^\pm_\e$ are parametrized counter-clockwise.
\end{itemize}
Then, since $h$ is holomorphic in $U^+$ and $U^-$ respectively, we have
\begin{align}
\int_{R^\pm_\e} h =0.
\end{align}
Hence, by the continuity of $h$,
\begin{align}
0=\int_{R_\e^+} h + \int_{R_\e^-} h \xrightarrow{\e \to 0} \int_R h,
\end{align}
which implies 
\begin{align}
\int_R h =0
\end{align}
and therefore the claim holds.
\end{proof}

Next, we prove that either $Q$ vanishes identically or has isolated zeroes. In this latter case, the zeroes of $Q$ are singular points of negative index of some continuous vector fields associated to $Q$.
\begin{proposition}\label{p_zeroesQ}
One (and only one) of the following holds
\begin{itemize}
\item $Q\equiv 0$ on $\sss^2$; 
\item $Q$ has isolated zeroes and the null directions of $Q$ determine two continuous line fields with isolated singularities of negative index.
\end{itemize}
\end{proposition}
\begin{proof}
Let $x_0$ be a zero of $Q$ and consider local conformal coordinates centred at $x_0$. By Remark \ref{rmk_P} and Lemma \ref{l_isolated_singularities}, either $P$ vanishes in a neighbourhood of $z=0$ or $P(z)=z^k G(z)$ for a certain $k\in \mathbb{N}$ and for $G$ continuous so that $G(0)\neq 0$, i.e. $0$ is an isolated zero for $P$ (equivalently, $x_0$ is an isolated zero for $Q$) of negative index $-\frac{k}{2}$ (see Remark \ref{rmk_index} below).
\end{proof}
\begin{remark}\label{rmk_index}
Let $Q$ be the symmetric bilinear form defined, in local conformal coordinates, as
\begin{align}
Q=P\ \dint{z}^2 + \overline{P}\ \dint{\overline{z}}^2=2\textnormal{Re}(P\ \dint{z}^2).
\end{align}
A null direction of $Q$ at $p$ is a tangent vector $0\neq v\in T_p \sss^2$ so that
\begin{align}
Q(v,v)=0.
\end{align}
By Lemma \ref{l_isolated_singularities}, in a neighbourhood of an isolated zero $z=0$ for $Q$ the function $P$ can be written as $P(z) = z^k f(z)$ with $f(0) \neq 0$. To study the behavior of the null directions near the origin, we use polar coordinates $z = r e^{i\theta}$ and represent the tangent vector $v$ by its phase $\phi$, i.e., $v = e^{i\phi}$. Asymptotically as $r \to 0$, the equation for the null directions $Q(v,v)=0$ is equivalent to studying
\begin{align}
\text{Re}\left( r^k e^{i k \theta} f(0) e^{2i\phi} \right) = 0.
\end{align}
Up to a rotation in the complex plane, we can assume without loss of generality that $f(0) \in \mathbb{R}^+$. So the equation reduces to $\cos(k\theta + 2\phi) = 0$, which yields the following relation for the phase:
\begin{align}
k\theta + 2\phi(\theta) = \frac{\pi}{2} + m\pi, \quad m \in \mathbb{Z}.
\end{align}
Solving for $\phi(\theta)$, we obtain:
\begin{align}
\phi(\theta) = -\frac{k}{2}\theta + \frac{\pi}{4} + \frac{m\pi}{2}.
\end{align}
The total variation $\Delta_{\Gamma} \phi$ of the phase $\phi(\theta)$ along a closed loop $\Gamma$ around the origin ($0 \leq \theta \leq 2\pi$) is given by:
\begin{align}
\Delta_{\Gamma} \phi = \phi(2\pi) - \phi(0) = -\frac{k}{2}(2\pi) = -k\pi.
\end{align}
Hence, the index $\mathfrak{I}_0(v)$ of the null direction field around the isolated zero is:
\begin{align}
\mathfrak{I}_0(v) = \frac{\Delta_{\Gamma} \phi}{2\pi} = -\frac{k}{2}.
\end{align}
\end{remark}

\begin{lemma}\label{l_ClaimE}
If $Q\equiv 0$, then $u\in \widetilde{\mathcal{C}}$.
\end{lemma}
\begin{proof}
The proof follows as in \cite[Claim E]{EM19}.
\end{proof}

We are ready to prove Theorem \ref{t_main}.

\begin{proof}[Proof of Theorem \ref{t_main}]
We want to prove that $Q\equiv 0$ and hence, by Lemma \ref{l_ClaimE}, $u\in \widetilde{\mathcal{C}}$. Thanks to Proposition \ref{p_zeroesQ}, we only have to prove that $Q$ cannot have isolated zeroes. Indeed, assume by contradiction that $Q$ has isolated zeroes, and hence (Proposition \ref{p_zeroesQ}) that the null directions of $Q$ determines two continuous line fields with isolated singularities of negative index. Since the Euler characteristic of $\sss^2$ is 2, by the Poincar\'e--Hopf Index Theorem we get a contradiction.
\end{proof}


\section{Proofs of Theorems \ref{t_general} and \ref{t_eigenvalue}} \label{Proofs of Theorems 1.4 and 1.5}

In the present section we construct the classes of candidate solutions associated to  problems \eqref{Eq-general} and \eqref{Eq-eigenvalue}. This implies that the bilinear form $Q$, given by \eqref{Eq-Q-def}, is well defined also in these settings. In the meantime, we will also show that these candidate solutions are radially symmetric with respect to a fixed point: thanks to Remark \ref{Rmk-Continuity-Q}, this allows us to prove a counterpart of Lemma \ref{l_Q_continuity} and hence to obtain a globally defined continuous symmetric bilinear form. Once  this is done, the remaining part of Section \ref{Sec-t-main} can be retraced verbatim (notice that the explicit equation is never used), and hence we can obtain the proofs of Theorem \ref{t_general} and Theorem \ref{t_eigenvalue}.

The construction of the candidate solutions is very similar to that performed in \cite[Section 3]{EM19}. However, in our case, we have to take into account the fact that our problems \eqref{Eq-general} and \eqref{Eq-eigenvalue} can be decoupled into  two problems with the compatibility condition as  problem \eqref{Eq-main} is decoupled into problem \eqref{Eq-decoupled problems for W}  with \eqref{Eq-compatibility}.

\subsection{Candidate solutions for (\ref{Eq-general})}
We start by defining the following family of candidate solutions
\begin{align}
\mathcal{C}:=\{V:=(v^+,v^-)=(v^+_{p_+,w_+,a_+}, v^-_{p_-,w_-,a_-})\}_{(p_\pm,w_\pm,a_\pm)\in T\sss^2\times \rr^\pm}
\end{align}
solving respectively
\begin{align}
\begin{cases}
-\Delta v^+=f(v^+) & \inn B_{R^+}(q^+),\\
-\Delta v^-=-\gamma f\left(-\frac {v^-}\gamma\right) & \inn B_{R^-}(q^-),\\
v^\pm=0 & \onn \partial B_{R^\pm}(q^\pm),\\
p_\pm\in \overline{B_{R^\pm}(q^\pm)},\\
v^\pm(p_\pm)=a_\pm,\\
\nabla v^\pm(p_\pm)=w_\pm
\end{cases}
\end{align}
for some $R^\pm=R^\pm(p_\pm,w_\pm,a_\pm)>0$ and $q^\pm=q^\pm (p_\pm,w_\pm,a_\pm)\in \sss^2$. To simplify notation, the subscript $(p_\pm, w_\pm, a_\pm)$ is suppressed unless necessary for clarity. The existence of such functions is proved in \cite[Section 3.2]{EM19}.

Consider the family $\{W_r\}_r$ of solutions to the following ODE problem
\begin{align}\label{Eq-W_r}
\begin{cases}
W_r''(s)+\cot(s)W_r'(s)+f(W_r(s))=0 & \inn (0,r),\\ \ W_r(r)=0,\\ \ W_r\in C([0,r]).
\end{cases}
\end{align}

\begin{remark}\label{Rmk_gamma_W'}
Notice that since
\begin{align}
W'_r(s)=-\frac{1}{\sin(s)}\int_0^s \sin(z)\ f(W_r (z))\ \dint{z},
\end{align}
then, for every $r>0$,
\begin{align}
\lim_{s\to 0^+} W'_r(s) = 0
\end{align}
and, denoted $l_r:=\lim_{s\to r^-} W'_r(s)<0$,
\begin{align}
\lim_{r\to 0^+} l_r=0 \quad \andd \quad \lim_{r\to \pi^-} l_r =-\infty.
\end{align}
Hence, there always exists $R>0$ so that $\gamma=\gamma(R):=\frac{W_R'(R)}{W_{\pi-R}'(\pi-R)}$.
\end{remark}

Similarly to the torsion problem, given $R>0$ so that $\gamma=\frac{W_R'(R)}{W_{\pi-R}'(\pi-R)}$, we denote
\begin{align}
\widetilde{\mathcal{C}}=\Big\{V\in \mathcal{C}\ :\ R^+=R,\ R^-=\pi-R\ \textnormal{and}\ q^\pm\ \textnormal{are antipodal} \Big\}.
\end{align}
Fixed $o\in \sss^2$, we define
\begin{align}
v^+(x):= W_R(d^{\sss^2}(o,x)) \quad \andd \quad v^-(x):=-\gamma W_{\pi-R}(\pi-d^{\sss^2}(o,x))
\end{align}
and hence the function
\begin{align}
v:=v^+ \chi_{B_R(o)}+v^- \chi_{\sss^2\setminus \overline{B_R(o)}}
\end{align}
is a radial solution to
\begin{align}\label{Eq_v_gamma_BR}
\begin{cases}
-\Delta v=f(v)\chi_{B_R(o)}-\gamma f\left(-\frac{v}{\gamma}\right)\chi_{\sss^2\setminus \overline{B_R(o)}} & \inn \sss^2,\\ v=0 & \onn \partial B_R(o).
\end{cases}
\end{align}
Observe that by definition of $\gamma$
\begin{align}
\nabla v^+=\nabla v^- \quad \onn \partial B_R(o).
\end{align}
Hence, thanks to Remark \ref{Rmk_gamma_W'}, for any $x\in \partial \Omega$ and for every $w\in T_x \sss^2$ there always exists $r>0$ so that $x\in \partial B_r(o)$, for a certain $o\in \sss^2$, and so that the solution $v$ to \eqref{Eq_v_gamma_BR} satisfies
\begin{align}
\nabla v^+(x)=w.
\end{align}

\subsection{Candidate solutions for (\ref{Eq-eigenvalue})}

We start by defining the following family of candidate solutions
\begin{align}
\mathcal{C}:=\{V:=(v^+,v^-)=(v^+_{p_+,w_+,a_+}, v^-_{p_-,w_-,a_-})\}_{(p_\pm,w_\pm,a_\pm)\in T\sss^2\times \rr^\pm}
\end{align}
where, denoted $\lambda^+:=\lambda(\Omega^+)$ and $\lambda^-:=\lambda(\Omega^-)$, the functions $v^\pm$ satisfy
\begin{align}
\begin{cases}
-\Delta v^\pm=\lambda^\pm v^\pm & \inn B_{R^\pm}(q^\pm),\\
v^\pm=0 & \onn \partial B_{R^\pm}(q^\pm),\\
p_\pm\in \overline{B_{R^\pm}(q^\pm)},\\
v^\pm(p_\pm)=a_\pm,\\
\nabla v^\pm(p_\pm)=w_\pm
\end{cases}
\end{align}
for some $q^\pm=q^\pm(p_\pm,w_\pm,a_\pm)\in \sss^2$, where $R^\pm>0$ are such that the first Dirichlet eigenvalues of $B_{R^\pm}$ equal $\lambda^\pm$ respectively. In particular, $\pm v^\pm\geq 0$ are the first Dirichlet eigenfunctions of some geodesic balls $B_{R^\pm}(q^\pm)$ respectively. The construction of such a family can be pursued as in \cite[Section 3.1]{EM19}. To simplify notation, the subscript $(p_\pm, w_\pm, a_\pm)$ is suppressed unless necessary for clarity.

In this case, the classes $\mathcal{C}$ and $\widetilde{\mathcal{C}}$ coincide
\begin{align}
\widetilde{\mathcal{C}}:=\mathcal{C}.
\end{align}
Retracing what has been done in Section \ref{Sec-t-main}, what we get is that $u\in \mathcal{C}$ and hence
\begin{align}
\Omega^+=B_{R^+}(q^+) \quad \andd \quad \Omega^-=B_{R^-}(q^-).
\end{align}
Since $\Omega^+$ and $\Omega^-$ are complementary sets, we have $R^-=\pi-R^-$ and $q^+$ and $q^-$ are antipodal.


\bigskip

\noindent\textbf{Acknowledgements.} The first author has been supported by ``Centro di Ricerca Matematica Ennio De Giorgi'' and he is a member of GNAMPA-INdAM. The second author has been supported by JSPS KAKENHI Grant Numbers JP22K03381 and JP26K06856. This research was started while the first author visited Tohoku University; he wishes to thank its kind hospitality. 

\bigskip

\appendix

\section{Stochastic completeness for the two-phase heat equation}
Let $(M,g)$ be a Riemannian manifold with associated volume form $\textnormal{dv}$. Let $(\mathcal{E},\mathcal{F})$ be a Dirichlet form on $L^2(M,\textnormal{dv})$ with associated 
\begin{itemize}
\item positive semidefinite self-adjoint operator $L$ on $L^2(M,\textnormal{dv})$;
\item heat semigroup $\{T_t\}_t$.
\end{itemize}
We recall the following equivalent definitions of stochastic completeness for Dirichlet forms (see \cite{FOT94, MR12, St94})

\begin{definition} The Dirichlet form $(\mathcal{E},\mathcal{F})$ is \textit{stochastically complete} if one of the following equivalent conditions is satisfied
\begin{enumerate}
\item $T_t 1\equiv 1$ for some (equiv. all) $t>0$;
\item the unique bounded solution to
\begin{align}
\begin{cases}
v_t=-L(v) & \inn M\times (0,+\infty)\\
v=0 & \onn M\times \{0\}
\end{cases}
\end{align}
is the constant null function.
\end{enumerate}
\end{definition}

\begin{definition}The Riemannian manifold $(M,g)$ is said to be \textit{stochastically complete} if the Dirichlet form associated to the Laplace-Beltrami operator
\begin{align}
\mathcal{E}(u,w)=\int_{M} g(\nabla u, \nabla w) \dvol \quad \quad \textnormal{with domain $\mathcal{F}=H^1(M,\textnormal{dv})$}
\end{align}
is stochastically complete.
\end{definition}

In what follows, let $(M,g)$ be a stochastically complete Riemannian manifold of dimension $n\geq 2$. Fixed $\Omega \subset M$ of class $C^2$ and $\sigma:=\sigma^+ \chi_{\Omega}+\sigma^- \chi_{M\setminus \overline{\Omega}}$, where $0<\sigma^- <\sigma^+$ are constants, the present section is aimed to prove the uniqueness to the the following Cauchy problem:
\begin{align}
\begin{cases}
v_t=\textnormal{div}(\sigma \nabla v) & \inn M \times (0,+\infty) \\ v=v_0 & \onn M \times \{0\}\\
v\in L^\infty (M\times (0,+\infty))
\end{cases}
\end{align}
for every $v_0\in L^\infty(M)$.

In particular, denoted by
\begin{align}
\mathcal{E}_\sigma (u,w):=\int_{M} \sigma g(\nabla u, \nabla w) \dvol
\end{align}
the Dirichlet form associated to the operator $-\textnormal{div}(\sigma\nabla \cdot)$, whose domain is
\begin{align}
\mathcal{F}:=H^1(M,\textnormal{dv}),
\end{align}
we are going to prove that under a suitable control on the volume growth of geodesic balls, which is a classical sufficient condition for the stochastic completeness of the manifold, also the stochastic completeness of the form $(\mathcal{E}_\sigma,\mathcal{F})$ holds.
\medskip

Following \cite{St94}, we denote the \textit{intrinsic metric} associated to $\mathcal{E}_\sigma$ and $\mathcal{E}$ respectively by
\begin{align}
\rho_\sigma :M\times M \to [0,+\infty) \quad \andd \quad \rho:M\times M\to [0,+\infty)
\end{align}
where
\begin{align}
\rho_\sigma(x,y):=\sup\left\{u(x)-u(y)\ :\ u\in H^1_{loc}(M,\textnormal{dv})\cap C(M),\ \sigma |\nabla u|^2 \leq 1\right\}
\end{align}
and
\begin{align}
\rho (x,y):=\sup\left\{u(x)-u(y)\ :\ u\in H^1_{loc}(M,\textnormal{dv})\cap C(M),\ |\nabla u|^2 \leq 1\right\}.
\end{align}

\begin{lemma}\label{lem_equivalent_metrics}
The metrics $\rho_\sigma$ and $\rho$ are equivalent
\begin{align}
\frac{1}{\sqrt{\sigma^+}}\rho\leq \rho_\sigma \leq \frac{1}{\sqrt{\sigma^-}} \rho.
\end{align}
Moreover, $\rho$ coincides with the Riemannian distance function $d^g$.
\end{lemma}
\begin{proof}
We start by observing that fixed $x,y\in M$ and any Lipschitz-continuous curve $\gamma:[0,1]\to M$ joining $x$ and $y$, for any $u\in H^1_{loc}(M,\textnormal{dv})\cap C(M)$ such that $\sigma |\nabla u|^2 \leq 1$ it holds
\begin{align}
u(x)-u(y)&=-(u(\gamma(1))-u(\gamma(0)))\\
&=-\int_0^1 \gamma'[u](\gamma(s))\ \dint{s}\\
&=-\int_0^1 g(\nabla u, \gamma')\ \dint{s}\\
&\leq \int_0^1|\nabla u| |\gamma'|\ \dint{s}\\
&\leq \frac{1}{\sqrt{\sigma^-}} \int_0^1 |\gamma'|\ \dint{s}
\end{align}
and hence
\begin{align}
u(x)-u(y)\leq \frac{1}{\sqrt{\sigma^-}} L[\gamma]\leq \frac{1}{\sqrt{\sigma^-}}d^g(x,y) 
\end{align}
where $L[\gamma]$ is the length of the curve $\gamma$. The same computation shows that for any $w\in H^1_{loc}(M,\textnormal{dv})\cap C(M)$ such that $|\nabla w|^2 \leq 1$
\begin{align}
w(x)-w(y)\leq d^g(x,y).
\end{align}
Hence
\begin{align}
\rho_\sigma(x,y)\leq \frac{1}{\sqrt{\sigma^-}}d^g(x,y) 
\end{align}
and
\begin{align}
\rho(x,y)\leq d^g(x,y).
\end{align}

Conversely, define
\begin{align}
u_y:=\frac{1}{\sqrt{\sigma^+}} d^g(y,\cdot) \quad \andd \quad w_y:=d^g(y,\cdot).
\end{align}
We have that $u_y\in H^1_{loc}(M,\textnormal{dv})\cap C(M)$ is such that $\sigma |\nabla u_y|^2 \leq 1$ and $w_y\in H^1_{loc}(M,\textnormal{dv})\cap C(M)$ is such that $|\nabla w|^2 \leq 1$. Hence
\begin{align}
\frac{1}{\sqrt{\sigma^+}} d^g(x,y)=u_y(x)-u_y(y)\leq \rho_\sigma(x,y)
\end{align}
and
\begin{align}
d^g(x,y)=u_y(x)-u_y(y)\leq \rho(x,y).
\end{align}

By the above computations it follows that
\begin{align}
\frac{1}{\sqrt{\sigma^+}} d^g(x,y)\leq \rho_\sigma(x,y) \leq \frac{1}{\sqrt{\sigma^-}} d^g(x,y)
\end{align}
and
\begin{align}
\rho (x,y)=d^g(x,y).
\end{align}
\end{proof}

As observed in \cite[Page 188, point (c)]{St94}, a first consequence of Lemma \ref{lem_equivalent_metrics} is that $(M, d^g)=(M,\rho)$ and $(M,\rho_\sigma)$ generate the same topology and the metric balls $B_R^\sigma(z):=\{x\in M\ :\ \rho_\sigma(z,x)<R\}$ associated to the metric $\rho_\sigma$ are relatively compact in $M$, since for every $z\in M$ and $R>0$
\begin{align}
B_{\sqrt{\sigma^-}R}(z)\subseteq B_{R}^\sigma (z)\subseteq B_{\sqrt{\sigma^+}R}(z).
\end{align}
In particular, it follows that, for any fixed $x\in M$, the functions
\begin{align}
V^\sigma(r):=|B_r^\sigma(x)|_{g} \quad \andd \quad V(r):=|B_r(x)|_{g}
\end{align}
satisfy, for every $R>0$,
\begin{align}\label{Eq_comparison_volumes_Dirichlet_forms}
V(\sqrt{\sigma^-} R)\leq V^\sigma(R)\leq V(\sqrt{\sigma^+}R).
\end{align}

\begin{theorem}\label{thm_equivalent_stochastic_completeness}
The following conditions are equivalent
\begin{itemize}
	\item $\int_1^{+\infty} \frac{r}{\log(V(r))}\ \dint{r}=+\infty$;
	\item $\int_1^{+\infty} \frac{r}{\log(V^\sigma(r))}\ \dint{r}=+\infty$.
\end{itemize}
Moreover, if one of (and hence all) the above conditions holds, then both $(\mathcal{E},\mathcal{F})$ and $(\mathcal{E}_\sigma,\mathcal{F})$ are stochastically complete.
\end{theorem}
\begin{proof}
By \eqref{Eq_comparison_volumes_Dirichlet_forms} it follows that
\begin{align}
\int_1^{+\infty} \frac{r}{\log(V(r))}\ \dint{r}=+\infty \quad \Leftrightarrow \quad \int_1^{+\infty} \frac{r}{\log(V^\sigma(r))}\ \dint{r}=+\infty.
\end{align}
Hence, if one of the above conditions holds, then both
\begin{align}
\int_1^{+\infty} \frac{r}{\log(V(r))}\ \dint{r}=+\infty \quad \andd \quad \int_1^{+\infty} \frac{r}{\log(V^\sigma(r))}\ \dint{r}=+\infty.
\end{align}
By \cite[Theorem 4]{St94}, both $(\mathcal{E},\mathcal{F})$ and $(\mathcal{E}_\sigma,\mathcal{F})$ are stochastically complete.
\end{proof}

To conclude, we obtain the following

\begin{corollary}\label{Cor_uniqueness_H^n}
Let $\Omega\subset \mathbb{H}^n$ be a $C^{2,\alpha}$ domain, for some $0<\alpha<1$. Set $\sigma:=\sigma^+ \chi_{\Omega}+\sigma^- \chi_{\mathbb{H}^n\setminus \overline{\Omega}}$, where $\sigma^+\neq \sigma^-$ are positive constants. Fixed any $v_0\in L^\infty(\mathbb{H}^n,\textnormal{dv})$, the Cauchy problem
\begin{align}
\begin{cases}
v_t=\textnormal{div}(\sigma \nabla v) & in\ \mathbb{H}^n\times (0,+\infty)\\
v=v_0 & on\ \mathbb{H}^n\times \{0\}
\end{cases}
\end{align}
admits at most one bounded solution.
\end{corollary}
\begin{proof}
In the particular case of $(M,g)=(\mathbb{H}^n,g^{\mathbb{H}^n})$, we have that
\begin{align}
V(r)\sim C_1 e^{(n-1)r} \quad \textnormal{as } r\to +\infty
\end{align}
for a positive constant $C_1$ depending on $n$. Hence
\begin{align}
\frac{r}{V(r)}\sim C_2\frac{r}{(n-1)r} =: C_3>0 \quad \textnormal{as } r\to +\infty,
\end{align}
where $C_3$ is constant, implying that
\begin{align}
\int_1^{+\infty} \frac{r}{V(r)}\ \dint{r}=+\infty.
\end{align}
By Theorem \ref{thm_equivalent_stochastic_completeness}, the claim follows.
\end{proof}

\bibliographystyle{alpha}
\bibliography{references}

@article{Al56,
  title={Uniqueness theorem for surfaces in the large V},
  author={Alexandrov, A. D.},
  journal={Vestnik Leningrad Univ.},
  volume={11},
  pages={5--17},
  year={1956}
}

@article{BBF98,
  title={Green's function, harmonic transplantation, and best {Sobolev} constant in spaces of constant curvature},
  author={Bandle, C. and Brillard, A. and Flucher, M.},
  journal={Transactions of the American Mathematical Society},
  volume={350},
  number={3},
  pages={1103--1128},
  year={1998},
  publisher={American Mathematical Society}
}

@article{BS25,
  title={Non-isoparametric {Serrin} domains of $\mathbb{S}^3$ with connected toric boundary},
  author={Bisterzo, A. and Sakaguchi, S.},
  journal={arXiv preprint arXiv:2511.16531},
  year={2025}
}

@article{CV17,
  author={Ciraolo, G. and Vezzoni, L.},
  title={A rigidity problem on the round sphere},
  journal={Commun. Contemp. Math.},
  volume={19},
  number={5},
  pages={1750001, 11 pp.},
  year={2017},
  doi={10.1142/S0219199717500018}
}

@article{CV19,
  author={Ciraolo, G. and Vezzoni, L.},
  title={On {Serrin's} overdetermined problem in space forms},
  journal={Manuscripta Math.},
  volume={159},
  number={3--4},
  pages={445--452},
  year={2019},
  doi={10.1007/s00229-018-1079-z}
}

@article{DNS25,
  title={Symmetry results for some overdetermined obstacle problems},
  author={De Nitti, N. and Sakaguchi, S.},
  journal={Proceedings of the American Mathematical Society},
  volume={153},
  number={7},
  pages={2919--2932},
  year={2025},
  publisher={American Mathematical Society}
}

@article{EM19,
  author={Espinar, J. M. and Mazet, L.},
  title={Characterization of {$f$}-extremal disks},
  journal={J. Differential Equations},
  fjournal={Journal of Differential Equations},
  volume={266},
  year={2019},
  number={4},
  pages={2052--2077},
  issn={0022-0396,1090-2732},
  mrclass={35N25 (35B50 35J25 35J91 35N10 58J05)},
  mrnumber={3906240},
  mrreviewer={Steven\ George\ Krantz},
  doi={10.1016/j.jde.2018.08.020},
  url={https://doi.org/10.1016/j.jde.2018.08.020}
}

@article{FM15,
  author={Fall, M. M. and Minlend, I. A.},
  title={Serrin's overdetermined problems on {Riemannian} manifolds},
  journal={Adv. Calc. Var.},
  volume={8},
  number={4},
  pages={371--400},
  year={2015}
}

@article{FMW,
  title={Serrin's overdetermined problem on the sphere},
  author={Fall, M. M. and Minlend, I. A. and Weth, T.},
  journal={Calculus of Variations and Partial Differential Equations},
  volume={57},
  pages={1--24},
  year={2018},
  publisher={Springer}
}

@book{FOT94,
  title={Dirichlet forms and symmetric {Markov} processes},
  author={Fukushima, M. and Oshima, Y. and Takeda, M.},
  volume={19},
  year={1994},
  publisher={W. de Gruyter Berlin}
}

@article{GM20,
  title={Uniqueness of immersed spheres in three-manifolds},
  author={G{\'a}lvez, J. A. and Mira, P.},
  journal={Journal of Differential Geometry},
  volume={116},
  number={3},
  pages={459--480},
  year={2020},
  publisher={Lehigh University}
}

@article{GNN79,
  author={Gidas, B. and Ni, Wei Ming and Nirenberg, L.},
  journal={Communications in Mathematical Physics},
  number={3},
  pages={209--243},
  title={Symmetry and related properties via the maximum principle},
  volume={68},
  year={1979}
}

@book{GT83,
  title={Elliptic partial differential equations of second order},
  author={Gilbarg, David and Trudinger, Neil S},
  year={1983},
  publisher={Springer}
}

@book{Jo91,
  title={Two-dimensional geometric variational problems},
  author={Jost, J.},
  year={1991},
  publisher={John Wiley and Sons, Ltd.}
}

@article{KP98,
  title={Serrin's result for hyperbolic space and sphere},
  author={Kumaresan, S. and Prajapat, J.},
  journal={Duke Mathematical Journal},
  volume={91},
  number={1},
  pages={17--28},
  year={1998}
}

@article{KS23,
  title={A symmetry theorem in two-phase heat conductors, in a special issue entitled ``{When} analysis meets geometry -- on the 50th birthday of {Serrin's} problem"},
  author={Kang, H. and Sakaguchi, S.},
  journal={Mathematics in Engineering},
  volume={5},
  number={3},
  pages={1--7},
  year={2023}
}

@article{M91,
  title={Symmetry and overdetermined boundary value problems},
  author={Molzon, R.},
  journal={Forum Mathematicum},
  volume={3},
  number={2},
  pages={143--156},
  year={1991},
  publisher={De Gruyter}
}

@book{MR12,
  title={Introduction to the theory of (non-symmetric) {Dirichlet} forms},
  author={Ma, Z. M. and R{\"o}ckner, M.},
  year={2012},
  publisher={Springer Science \& Business Media}
}

@article{NT18,
  author={Nitsch, C. and Trombetti, C.},
  title={The classical overdetermined {Serrin} problem},
  journal={Complex Variables and Elliptic Equations},
  volume={63},
  number={7--8},
  pages={1107--1122},
  year={2018}
}

@article{QX17,
  author={Qiu, G. and Xia, C.},
  title={Overdetermined boundary value problems in {$\mathbb S^n$}},
  journal={J. Math. Study},
  volume={50},
  number={2},
  pages={165--173},
  year={2017}
}

@article{R97,
  author={Reichel, W.},
  journal={Archive for Rational Mechanics and Analysis},
  number={4},
  pages={381--394},
  title={Radial symmetry for elliptic boundary-value problems on exterior domains},
  volume={137},
  year={1997}
}

@article{Se71,
  title={A symmetry problem in potential theory},
  author={Serrin, J.},
  journal={Archive for Rational Mechanics and Analysis},
  volume={43},
  pages={304--318},
  year={1971},
  publisher={Springer}
}

@article{St94,
  author={Sturm, K. T.},
  title={Analysis on local {D}irichlet spaces. {I}. {R}ecurrence, conservativeness and {$L^p$}-{L}iouville properties},
  journal={J. Reine Angew. Math.},
  fjournal={Journal f{\"u}r die Reine und Angewandte Mathematik. [Crelle's Journal]},
  volume={456},
  year={1994},
  pages={173--196},
  issn={0075-4102,1435-5345},
  mrclass={31C25 (31B05 60J45)},
  mrnumber={1301456},
  mrreviewer={Zhen-Qing\ Chen},
  doi={10.1515/crll.1994.456.173},
  url={https://doi.org/10.1515/crll.1994.456.173}
}

@article{We71,
  author={Weinberger, H. F.},
  title={Remark on the preceding paper of {Serrin}},
  journal={Archive for Rational Mechanics and Analysis},
  volume={43},
  number={4},
  pages={319--320},
  year={1971}
}

\end{document}